\documentclass[12pt,leqno]{amsart}
\usepackage{a4,amssymb,amsthm,amscd,amsmath,verbatim,url,enumerate,mathdots,hyperref,cleveref}
\usepackage[dvipsnames]{xcolor}
\usepackage[numbers]{natbib}

\newcommand{\qq}{\mathbb Q}

\newcommand{\nat}{\mathbb{N}}

\newcommand{\zz}{\mathbb Z}
\newcommand{\rr}{\mathbb R}
\renewcommand{\H}{\mathcal{H}}
\newcommand{\MH}{\mathsf{M}(\H)}

\newcommand{\Xcod}[1]{\mathcal{X}^{(#1)}}

\newcommand{\mc}[1]{\mathcal{#1}}

\newcommand{\mg}[1]{#1^{\times}}
\newcommand{\sq}[1]{#1^{\times 2}}

\renewcommand{\setminus}{\smallsetminus}
\newcommand{\ffps}[2]{#1 (\!(#2)\!)}
\newcommand{\rfps}[2]{#1 [\![#2]\!]}

\DeclareMathOperator{\newsum}{\mathsf{\Sigma}}
\renewcommand{\sum}{\newsum}
\newcommand{\sums}[1]{\sum\!{#1}^2}
\newcommand{\sos}[2]{\sum_{#1}\!{#2}^2} 

\newcommand{\follows}[2]{$(#1\Rightarrow #2)$}

\newcommand{\mfp}{\mathfrak{p}}

\newcommand{\mfm}{\mathfrak{m}}

\renewcommand{\inf}{\mathsf{inf}}
\renewcommand{\min}{\mathsf{min}}

\newcommand{\half}{\mbox{$\frac{1}2$}}
\newcommand{\X}{\mathcal{X}}

\newcommand{\Spec}[1]{\mathsf{Spec}(#1)}
\renewcommand{\dim}{\mathsf{dim}}

\newcommand{\mf}{\mathfrak}

\newcommand\restr[2]{{
  \left.\kern-\nulldelimiterspace 
  #1 
  \vphantom{\big|} 
  \right|_{#2} 
  }}
\newcommand{\car}[1]{\mathsf{char}({#1})}

\renewcommand{\leq}{\leqslant}
\renewcommand{\geq}{\geqslant}
\renewcommand{\setminus}{\smallsetminus}
\renewcommand{\bmod}{\,\,\mathsf{mod}\,\,}

\newcommand{\Max}[1]{\mathsf{Max}(#1)}

\newcommand{\Jac}{\mathsf{Jac}}
\newcommand{\M}{\mathsf{M}}
\newcommand{\trdeg}[2]{\mathsf{trdeg}(#2/#1)}

\swapnumbers
\numberwithin{equation}{section}
\newtheorem{thm}[equation]{Theorem}
\newtheorem*{thm*}{Theorem}
\newtheorem{prop}[equation]{Proposition}
\newtheorem{cor}[equation]{Corollary}
\newtheorem*{cor*}{Corollary}

\newtheorem*{conj*}{Conjecture}
\newtheorem{lem}[equation]{Lemma}

\newtheorem{qu}[equation]{Question}
\newtheorem*{qu*}{Question}

\theoremstyle{definition}

\newtheorem{ex}[equation]{Example}
\newtheorem{ex*}{Example}

\newtheorem{rem}[equation]{Remark}

\theoremstyle{plain}

\title[Bounding the Pythagoras number of a field by $2^n+1$]{Bounding the Pythagoras\\ number of a field by~$2^n+1$}
\date{10.02.2024}
\thanks{This work was supported by the FWO Odysseus programme (project G0E6114N, \textit{Explicit Methods in Quadratic Form Theory}).}
\author{Karim Johannes Becher}
\author{Marco Zaninelli}
\address{University of Antwerp, Department of Mathematics, Middelheim\-laan~1, 2020 Antwerpen, Belgium.}
\email{karimjohannes.becher@uantwerpen.be}
\email{marco.zaninelli@uantwerpen.be}

\begin{document}


\begin{abstract}

\medskip\noindent
Given a positive integer $n$, a sufficient condition on a field is given for bounding its Pythagoras number by $2^n+1$.
The condition is satisfied for $n=1$ by function fields of curves over iterated formal power series fields over $\rr$, as well as by finite field extensions of  $\rr(\!(t_0,t_1)\!)$. 
In both cases, one retrieves the upper bound $3$ on the Pythagoras number.
The new method presented here might help to establish more generally $2^n+1$ as an upper bound for the Pythagoras number of function fields of curves over $\rr(\!(t_1,\dots,t_n)\!)$ and for finite field extensions of $\rr(\!(t_0,\dots,t_n)\!)$.


\medskip\noindent
{\sc Classification (MSC 2020):} 11E25
\end{abstract}

\maketitle

\section{Introduction}

In 1964, J.W.S.~Cassels \cite[Theorem~2]{Cas64} proved that $X_0^2+\dots+X_n^2$ is not a sum of $n$ squares in the field $\rr(X_0,\dots,X_n)$ for any $n\in\nat$.
This was the starting point of a systematic study of sums of squares in fields, more specifically of the problem of how many terms are generally needed in a given field to represent an arbitrary positive element as a sum of squares.
This is captured in the notion of Pythagoras number, which was introduced  by L.~Br\"ocker \cite{Bro78} and A.~Prestel \cite{Pre78}.

Given a commutative ring $R$, we denote by $\sos{k}{R}$ for $k\in\nat$ the set of elements of $R$ that are sums of $k$ squares in $R$, and we set $\sums{R}=\bigcup_{k\in\nat}\sos{k}{R}$.

Let $F$ be a field. 
For $S\subseteq F$ we set $\mg{S}=\{x\in S\setminus\{0\}\mid x^{-1}\in S\}$.
We have $\mg{(\sos{k}{F})}=\sos{k}{F}\setminus\{0\}$ for any $k\in\nat$. 
Furthermore, $\mg{(\sums{F})}=\sums{F}\setminus\{0\}$ and this is a subgroup of $\mg{F}$.
A.~Pfister \cite[Satz 2]{Pfi65} proved that $\mg{(\sos{2^n}{F})}$ is a subgroup of $\mg{F}$ for any $n\in\nat$.
By the Artin-Schreier Theorem \cite[Satz 7b]{AS26}, the field $F$ admits a field ordering if and only if $-1\notin\sums{F}$.
In view of this fact, the field $F$ is called \emph{real} if $-1$ is not a sum of squares in $F$, and  \emph{nonreal} otherwise.
The values
\begin{eqnarray*}
p(F) & = & \inf\{k\in\nat\mid \hbox{$\sums{F}$}=\sos{k}{F}\}\, \,\in\nat\cup\{\infty\} \mbox{ and }\\
s(F) & = & \inf\{k\in\nat\mid -1\in \sos{k}{F}\} \quad \in \nat\cup\{\infty\}\,
\end{eqnarray*}
are called the \emph{Pythagoras number} and the \emph{level of $F$}, respectively.
Note that $s(F)=\infty$ if and only if $F$ is real.
If $F$ is nonreal, then as a consequence of the fact that $\mg{(\sos{2^n}{F})}$ is a group for any $n\in\nat$, A.~Pfister \cite[Satz 4]{Pfi65} showed that $s(F)$ is a power of $2$, and it is easy to see that $s(F)\leq p(F)\leq s(F)+1$.
Examples of nonreal fields can be constructed to realise all pairs of values for $s(F)$ and $p(F)$ subject to these constraints. 

D. Hoffmann showed that every positive integer occurs as the Pythagoras number of some real field \cite{Hof99}. 
Nevertheless, to this date the only known examples of real fields with Pythagoras number not contained in $\{2^n,2^n+1\mid n\in\nat\}\cup\{\infty\}$ are constructed by an infinite iteration of function field extensions. 
In particular, no such fields are known that are finitely generated over any proper subfield.

A crucial open problem is to understand the growth of the Pythagoras number under field extensions with a real base field. 
We denote by $F(X)$ the rational function field in the variable $X$ over $F$.
It has been a crucial discovery by J.W.S.~Cassels \cite[Theorem 2]{Cas64} that $p(F(X))>p(F)$ holds whenever $F$ is real.
However, we still do not know whether the Pythagoras number grows slowly or fast when adding variables and passing from a real field to the rational function field over it. 
To this date, there is no confirmed example where $p(F(X))>p(F)+2$.
For arbitrary $n\in\nat$, A.~Pfister showed in \cite[Satz 2]{Pfi67} that $p(F(X))\leq 2^{n+1}$ if and only if $s(F')\leq 2^n$ for all finite nonreal field extensions $F'/F$.

Let us turn our attention to the problem of determining the Pythagoras number of a specific field. 
Let $n\in\nat$.
A famous result due to A.~Pfister states that $p(F)\leq 2^n$ when $F$ is a field extension of transcendence degree $n$ of $\rr$ (or of any real closed field); see \cite[Theorem 1]{Pfi67}.
In particular 
$p(\rr(X_1,\dots,X_n))\leq 2^n$.
J.W.S.~Cassels, W.J.~Ellison and A.~Pfister proved in \cite{CEP71} that $p(\rr(X_1,X_2))= 4$ by showing that the polynomial $$M(X_1,X_2)=X_1^2X_2^4+X_1^4X_2^2-3X_1^2X_2^2+1$$ is not a sum of three squares in $\rr(X_1,X_2)$.
The polynomial $M$ is known as the \emph{Motzkin polynomial}, after T.S.~Motzkin, who gave it in \cite{Mo67} as an example of a polynomial that 
is a sum of squares in the field $\rr(X_1,X_2)$, but not in the polynomial ring $\rr[X_1,X_2]$.

The fact that the result of \cite{CEP71} and Cassels' observation on $X_0^2+\dots+X_n^2$ mentioned at the beginning were not found before the 1960s may indicate how difficult it is to show that a certain sum of squares in a field cannot be equal to a shorter sum of squares.

O.~Benoist recently showed that  $$p\big(\rr(\!(X_0,\dots,X_n)\!)(Y_1,\dots,Y_r)\big)\leq 2^{n+r}\,$$ 
for any $n,r\in\nat$; this follows from \cite[Theorem 0.2]{Ben20} by using \cite[Theorem 3.5]{BVG09}.
When $n\geq 1$ and $F$ is a finite extension of $\rr(\!(X_0,\dots,X_n)\!)$, then the statement for $r=1$ implies that $p(F)< 2^{n+1}$.
One may ask whether this can be improved.

\begin{qu}\label{Q1}
\!Is $p(F)\leq 2^n+1$ for every finite field extension $F/\rr(\!(X_0,\dots,X_n)\!)$?
\end{qu}

This bound is trivial and optimal for $n=0$, and for $n=1$ this bound was shown in \cite[Theorem 5.1]{Hu15}.
 Clearly \Cref{Q1} has a positive answer when restricting to nonreal fields $F$, because then $p(F)<2^{n+1}$ implies that $s(F)\leq 2^n$ and thus $p(F)\leq 2^n+1$; see also \cite[Theorem 0.2~$(i)$]{Ben20}.
The question is open for $n\geq 2$ when $F$ is real.

In this article, we develop a new method that may help to obtain a positive answer to \Cref{Q1}. 
The method is specifically designed for showing under certain  assumptions on a field $F$ that $p(F)\leq 2^n+1$.
These assumptions include the presence of a local-global principle for certain quadratic forms.
As such, this is a fairly standard approach. 
Indeed, given $k\in\nat$ and an arbitrary field element $a$ which is a sum of squares, the question whether $a$ can be written as a sum of $k$ squares is reformulated in terms of the existence of a nontrivial zero for the $(k+1)$-ary quadratic form $\sum_{i=1}^kX_i^2-aX_0^2$, and hence a local-global criterion for the existence of zeros may help in bounding by $k$ the number of square terms that are needed. 
In our approach, however, we assume that $k=2^n$ and we want to use a valuation theoretic local-global principle for characterising the sums of $k$ squares (rather than the sums of $k+1$ squares) while aiming for the bound $p(F)\leq k+1$.
This is reminiscent of Pourchet's method in \cite{Pou71} for proving that $p(K(X))\leq 5$ when $K$ is a number field, where a description of the sums of $4$ squares in $K(X)$ via local conditions is essential.
Field extensions $F/\qq$ of transcendence degree one provide an example where a local-global criterion for sums of $4$ squares in $F$ is available by \cite[Theorem 0.8~$(2)$]{Kato86}, while a similar local-global criterion for being a sum of $5$ squares would actually imply that $p(F)\leq 5$, which however is still an open problem.

Our method for obtaining the bound $p(F)\leq 2^n+1$ depends on a certain subring $\H$ of $F$ that provides a characterisation of the sums of $2^n$ squares in $F$.
After developing the new setup in the next two sections, we will turn in the last two sections to applications, first to function fields in one variable, then to finite extensions of a field of formal power series in two variables, in both cases under strong hypotheses on the base field.
By a \emph{function field in one variable}, we mean a finitely generated field extension of transcendence degree one.

In our applications the subring $\H$ of $F$ will be given as a finite intersection of valuation rings of $F$, where the corresponding valuations yield a local-global principle for the quadratic forms $\sum_{i=1}^{2^n}X_i^2-aX_0^2$ with $a\in\mg{(\sums{F})}$.
These quadratic forms are in particular Pfister neighbors, and due to the direct link of such forms to Galois cohomology classes, such a local-global principle may be easier to establish than for the corresponding quadratic forms of dimension $2^n+2$. Up to this local-global ingredient, our method is elementary.

As an illustration of the efficiency and applicability of our method, let us look at a well-known example, due to S.~Tikhonov. Let $F$ be the function field of the curve  $Y^2=(tX-1)(X^2+1)$ over  $\rr(\!(t)\!)$.
It was observed in \cite[Example 3.10]{TVGY06} that $3\leq p(F)\leq 4$, and it was shown in \cite[Corollary 6.13]{BGVG14} that $p(F)=3$.
The argument in \cite{BGVG14} for having $p(F)\leq 3$ relies on a deep local-global principle from \cite[Theorem 3.1]{CTPS12}, based on field patching and using further results from algebraic geometry, such as embedded resolutions of singularities.
In \Cref{Tikhonov curve}, we will retrieve the equality $p(F)=3$ from our new method and results from \cite{BVG09}. This approach uses no results from algebraic geometry and only Milnor's Exact Sequence for the rational function field $\rr(\!(t)\!)(X)$ as a local-global ingredient (via the proof of \cite[Theorem 3.10]{BVG09}).

Along the same lines, we will obtain in \Cref{Motzkin curve} that $p(F)=5$ for the field
$$F=\rr (Y)(\!(t)\!)(X)\left(\sqrt{(tX-1)M(X,Y)}\right)$$ where $M(X,Y)=X^2Y^4+X^4Y^2-3X^2Y^2+1$, the Motzkin polynomial from above.

Our study will also lead to the following result, which seems not to be covered in the literature, although it could also be proven by methods already in place.

\begin{thm}[\Cref{C:final bound}]\label{T:new-bound}
Let $n\in\nat^+$ be such that $p(E)\leq 2^{n}$ for every finite field extension $E/K(X)$. 
Let $r\in\nat$ and let $F/K(\!(t_1)\!)\dots(\!(t_r)\!)(\!(t_{r+1},t_{r+2})\!)$ be a finite field extension.
Then $p(F)\leq 2^n+1$.
\end{thm}
This recovers the positive answer to \Cref{Q1} for $n=1$, which was previously obtained in \cite[Theorem 5.1]{Hu15} by different means.
\Cref{T:new-bound} also applies to the cases where $K$ is an extension of transcendence degree $n-1$ of $\rr$, or an extension of transcendence degree $n-2$ of $\qq$ for $n\geq 3$; see \Cref{FPS}.
So, for example, we now obtain that $p(F)\leq 9$ for every 
finite extension $F$ of $\qq(X)(\!(t_1,t_2)\!)$.
\smallskip

We end the introduction by fixing some terms for a general commutative ring $R$.
We denote by $\mg{R}$ the group of invertible elements in $R$, by $\car R$ the characteristic of $R$, by $\Max R$ the set of maximal ideals of $R$ and by $\Spec R$ the set of prime ideals of $R$. 
We say that $R$ is \emph{semilocal} if $|\Max R|<\infty$. We denote by $\Jac(R)$ the \emph{Jacobson radical of $R$},  defined by $$\Jac(R)=\{x \in R \mid 1-Rx\subseteq \mg{R}\}\,,$$ and we recall that $\Jac(R)=\bigcap \Max R$. 
We denote by $\textbf{0}_R$ the zero ideal of $R$.
For a prime ideal $\mfm\in\Spec R$, we denote by $R_\mfm$ the localisation of $R$ at $\mfm$.

Furthermore, we set $\nat^+=\nat\setminus\{0\}$.
\pagebreak

\section{Square-effective rings} 

Let $F$ denote a field of characteristic different from $2$.

Let $\H$ be a subring of $F$ having $F$ as its field of fractions.
We set $$p^*(\H)=\inf\{ k\in \nat \mid \mg \H \cap \sums F  \subseteq \sos{k}{F} \} \in \nat \cup \{\infty\}\,.$$ 
Note that if $F$ is nonreal, then $-1\in \mg \H \cap \sums F$, thus $ s(F)\leq p^*(\H)$.

\begin{prop}\label{T:main}
Assume that $\sums F\subseteq F^2\cdot   ((1+\sums F) \cap\Jac(\H))$.
Then $$p(F)\leq p^*(\H)+1\,.$$
\end{prop}

\begin{proof}
We set  $p=p^*(\H)$.
Consider $g\in \sums{F}$. 
By the hypothesis, there exist $c\in F$ and $f \in (1+\sums F) \cap \Jac(\H)$ such that $g=c^2f$.
Since $f\in \Jac(\H)$, we have $1-4f\in \mc H^{\times}$, and since $f \in 1+\sums F$, we have $4f-1 \in 3 + 4\sums F \subseteq \sums{F}$. Hence $4f-1 \in \mathcal{H}^{\times} \cap \sums{F}\subseteq \sos{p}{F}$.
We conclude that $$g-(\half c)^2=(\half c)^2(4f-1)\in \sos{p}{F}\,,$$ whereby $g\in \sos{p+1}{F}$.
This shows that $\sums{F}=\sos{p+1}{F}$.
\end{proof}

We say that $\H$ is \emph{square-effective} if, for every $f_1,f_2 \in F$, there exist $g_1,g_2 \in F$ such that $f_1^2+f_2^2=g_1^2+g_2^2$ and $f_1\mc H + f_2\mc H = g_1 \mc H \supseteq g_2\H $.
A \emph{B\'ezout domain} is a domain in which every finitely generated ideal is principal.
Recall that B\'ezout domains are integrally closed \cite[Theorem 50]{Kap74}.

\begin{prop}\label{Minimal Value Representation}
Assume that $\H$ is square-effective. Let $k\in\nat^+$. Then: 
\begin{enumerate}[$(a)$]
\item $\H$ is a B\'ezout domain.
\item For every $f_1,\dots,f_k \in F$, there exist $g_1,\dots,g_k \in F$ such that\\ $f_1^2+\dots+f_k^2=g_1^2+\dots+g_k^2$ and $f_1\mc H+\dots + f_k\mc H = g_1 \mc H \supseteq g_2\H \supseteq \dots \supseteq g_k\mc H$.
\item $\sos{k}F= F^2 \cdot (1+ \sos{k-1}\H)$.
\item If $1+\sos{k-1}\H\subseteq \mg\H$, then $\H\cap\sos{k}F= \sos{k}\H$.
\item If $2\in \mg \H$, then for every $\mfm\in\Max\H$, we have $\mg{\H}\cap \sos{k}{F}\subseteq \sos{k}{\H}+\mfm$.
\end{enumerate}
\end{prop}

\begin{proof}
$(a)$\ Since $\H$ is square-effective, any ideal of $\H$ generated by two elements is principal. Therefore any finitely generated ideal of $\H$ is principal.

$(b)$\
We prove the statement by induction on $k$.
If $k=1$, then there is nothing to show.
Assume now that the statement holds for some $k \geq 1$. 
In order to prove the statement for $k+1$, we consider $f_0,\dots, f_k \in \H$.
By the induction hypothesis, there exist $h_1,\dots,h_k \in F$ such that $f_1^2+\dots+f_k^2=h_1^2+\dots+h_k^2$ and $f_1\mc H+\dots + f_k\mc H = h_1 \mc \H \supseteq \dots \supseteq h_k\mc H$.
It follows that $f_0\H+\dots +f_k\H=f_0\H+h_1\H$. As $\mc H$ is square-effective, we may choose $g_0,h \in F$ such that $f_0^2+h_1^2=g_0^2+{h^2}$ and $f_0\mc H+ h_1\mc H=g_0\H \supseteq h\H$.
Then we have $f_0^2+\dots+f_k^2=g_0^2+{h^2}+h_2^2+\dots+h_k^2$ and $f_0\mc H+\dots + f_k\mc H=g_0\H \supseteq h\H +h_2\H \dots + h_k\mc H$.
Again by the induction hypothesis, there exist $g_1,\dots,g_k\in F$ such that 
${h^2}+h_2^2+\dots+h_k^2=g_1^2+\dots+g_k^2$ and $h\H+h_2\H+\dots+h_k\H=g_1\H\supseteq g_2\H\supseteq \dots\supseteq g_k\H$.
We conclude that 
 $f_0^2+\dots+f_k^2=g_0^2+\dots+g_k^2$ and $f_0\H+\dots+f_k\H=g_0\H\supseteq g_1\H\supseteq \dots\supseteq g_k\H$.

$(c)$\ It follows by $(b)$ that $\sos{k}F\subseteq F^2 \cdot (1+ \sos{k-1}\H)$. 
The opposite inclusion holds trivially.

$(d)$\ 
Let $f\in \H\cap\sos{k}F$.
By $(c)$, we obtain that $f=x^2\cdot h$ for certain 
$x \in F$ and $h \in 1+ \sos{k-1}\H$. 
Now, if $h\in\mg \H$, then as $x$ is a root of $X^2-fh^{-1}$ and $\H$ is integrally closed, we get that $x\in \H$, whereby $f\in \H^2\cdot (1+ \sos{k-1}\H)\subseteq \sos{k}{\H}$.
Therefore, if $(1+ \sos{k-1}\H)\subseteq \mg \H$, then $\H\cap\sos{k}F= \sos{k}\H$.

$(e)$\ Let $\mfm\in\Max\H$.
Since $\H$ is square-effective, so is $\H_\mfm$, and we have $2\in\mg\H$. 
Assume first that $s(\H/\mfm)\geq k$.
Since $\H/\mfm\simeq\H_\mfm/\mfm\H_\mfm$, it follows 
that $1+\sos{k-1}{\H}_\mfm\subseteq \mg{\H}_\mfm$.
Hence, by $(d)$, we have $\H_\mfm\cap \sos{k}{F}=\sos{k}{\H}_\mfm$.
Again using that $\H/\mfm\simeq\H_\mfm/\mfm\H_\mfm$, we conclude that $\H\cap \sos{k}F\subseteq \sos{k}\H +\mfm$.
Assume now that $k> s(\H/\mfm)$.
Then $-1 \in \sos{k-1}{\H}+\mfm$.
Using  this and the identity $x=\left(\frac{x+1}{2}\right)^2 + (-1)\left(\frac{x-1}{2}\right)^2$ for $x\in\H$, we
conclude that $\H\subseteq \sos{k}\H+\mfm$.
\end{proof}

\begin{thm}\label{C:square-effective-Mnotempty}
Assume that $\H$ is square-effective and $(1+\sums F) \cap\Jac(\H)\neq \emptyset$. 
Then $$p(F)\leq p^*(\H)+1\,.$$
\end{thm}
\begin{proof}
Set $\mathsf{M}(\H)=  (1+\sums F)\cap\Jac(\H)$.
If $0\in\mathsf{M}(\H)$, then $F$ is nonreal, and hence $p(F)\leq s(F)+1\leq p^*(\H)+1$.
Assume now that $0\notin\mathsf{M}(\H)$.
As $\mathsf{M}(\H)\neq \emptyset$ by the hypothesis, we may fix an element $g\in \MH$. 
Then $g^2 \in F^{\times 2} \cap \MH$, whereby $F^2=F^2g^2\subseteq F^{2} \cdot \MH$.
Since $\MH \cdot (1+ \sums{\mc H}) \subseteq \MH$, we conclude that $F^2 \cdot (1+ \sums{\mc H}) \subseteq F^2 \cdot \MH \cdot (1+ \sums{\mc H}) \subseteq F^2 \cdot \MH$.
By \Cref{Minimal Value Representation} $(c)$,
we obtain that
$\sums{F}\subseteq F^2\cdot (1+\sums\H)\subseteq F^2\cdot \M(\H)$. 
Now the statement  follows by \Cref{T:main}.
\end{proof}

We will mainly consider the condition that $\H$ is square-effective in combination with the condition that $1+\H^2\subseteq\mg{\H}$.

\begin{lem}\label{Level One}
Assume that $1+\H^2 \subseteq \mg\H$.
Then $f^2+g^2\in\mg{\H}$ for every $f,g \in \mc H$ with $f\mc H+ g\mc H=\H$. 
\end{lem}

\begin{proof}
Consider $f,g \in \H$ such that $f^2+g^2\notin\mg\H$.
Then $f^2+g^2\in\mfm$ for some $\mfm \in \Max \H$.
If $f\notin \mfm$, then there exists $h\in\H$ such that $fh\equiv 1\bmod \mfm$, whereby $1+(gh)^2\equiv h^2(f^2+g^2)\equiv 0\bmod \mfm$,  which contradicts the hypothesis that $1+\H^2\subseteq\mg{\H}$.
Therefore $f\in\mfm$, and similarly we obtain that $g\in \mfm$.
Hence $f\H+g\H\subseteq\mfm$ and in particular $f\H+g\H\neq \H$.
\end{proof}

\begin{prop}\label{Square-effective k=2}
The following are equivalent:
\begin{enumerate}[$(i)$]
\item $\mc H$ is square-effective and $1+\H^2 \subseteq \mg\H$.
\item $\H$ is a B\'ezout domain and $\mg\H \cap (\H^2+\H^2)= \sq \H(1+\H^2)$.
\end{enumerate}
\end{prop}

\begin{proof} 
\follows{i}{ii}
Condition $(i)$ clearly implies that  $\sq \H(1+\H^2)\subseteq\mg\H \cap (\H^2+\H^2)$ and, by \Cref{Minimal Value Representation}, that $\H$ is a B\'ezout domain. It remains to show that $\mg\H\cap(\H^2+\H^2)\subseteq \sq \H(1+\H^2)$. Consider $f \in  \mg\H\cap(\H^2+\H^2)$.
Since $\H$ is square-effective, there exist $g_1,g_2 \in F$ such that $f=g_1^2+g_2^2$ and $g_1 \mc H \supseteq g_2\H$. 
Then $f=g_1^2(1+(g_1^{-1}g_2)^2)$, and $g_1^{-1}g_2\in \H$.
Since $1+\H^2 \subseteq \mg\H$, we have that $f\in g_1^2\mg \H \cap \mg{\H}$, whereby $g_1^2\in \mg \H$.
Since $\H$ is integrally closed, we conclude that
$g_1\in\mg{\H}$, whereby $f=g_1^2(1+(g_1^{-1}g_2)^2)\in\sq{\H}\cdot (1+\H^2)$.

\follows{ii}{i}
Assume that $\H$ is a B\'ezout domain and $\mg\H \cap (\H^2+\H^2)= \sq \H(1+\H^2)$.
Then clearly $1+\H^2 \subseteq \mg\H$.
It remains to show that $\H$ is square-effective.

Consider $f_1,f_2 \in F$, not both equal to zero.
Since  $\H$ is a B\'ezout domain and its fraction field is $F$,
there exists $g\in\mg{F}$ such that $f_1\H+f_2\H=g\H$.
Then $f_1g^{-1}, f_2g^{-1} \in \mc H$ and $\mc H = f_1g^{-1}\mc H+ f_2g^{-1}\mc H$.
Since $1+\H^2 \subseteq \mg\H$, it follows by \Cref{Level One} that $f_1^2g^{-2}+ f_2^2g^{-2} \in \mg{\mc H}$. 
Since $\mg\H \cap (\H^2+\H^2)\subseteq \sq \H(1+\H^2)$, we may choose $l_1 \in \mc H^{\times}$, $l_2 \in \mc H$ such that $f_1^2g^{-2}+ f_2^2g^{-2}=l_1^2+l_2^2$.
Letting $g_1=gl_1$ and $g_2=gl_2$, we obtain that $f_1^2+f_2^2=g_1^2+g_2^2$ and $f_1\mc H+ f_2\mc H=g_1\H \supseteq g_2\H$.
This shows that $\H$ is square-effective.
\end{proof}

\begin{lem}\label{Value Zero}
Assume that $1+\H^2 \subseteq \mg\H$. Let $\mc{M}$ be a finite subset of $\Max \H$ and  $f_1,f_2\in \H$ such that $f_1^2+f_2^2 \notin \bigcup\mc{M}$. 
Then there exist $g_1\in \H\setminus\bigcup\mc{M}$ and $g_2\in \H$ such that $f_1^2+f_2^2=g_1^2+g_2^2$.
\end{lem}

\begin{proof}
Consider $G_1=(X^2-1) f_1+2X f_2$ and $G_2=2Xf_1+(1-X^2)f_2$ in $\mc{H}[X]$ and observe that $G_1^2+G_2^2=(1+X^2)^2\cdot(f_1^2+f_2^2)$.
Consider $\mfm\in M$.
Since $1+\mc{H}^2\subseteq\mg{\mc{H}}$, we have $\mathsf{char}(\mc{H}/\mfm)\neq 2$, whereby $|\mc{H}/\mfm|>2$. Hence there exists $x_{\mfm}\in\mc{H}$ such that $G_1(x_{\mfm})\notin \mfm$.
Since $\mc{M}$ is finite, by the Chinese Remainder Theorem, there exists $x\in\mc{H}$ with $x\equiv x_{\mfm}\bmod \mfm$ for all $\mfm\in M$, whereby $G_1(x)\in \mc{H}\setminus \bigcup 
\mc{M}$.
We now set $g_i=(1+x^2)^{-1}G_i(x)$ for $i=1,2$ to obtain the desired conclusion.
\end{proof}

\begin{prop}\label{Semilocal case}
Assume that $\H$ is semilocal and $1+\H^2 \subseteq \mg\H$. Then  $$\mg\H\cap(\H^2+\H^2)   =\sq \H(1+\H^2)\,.$$
\end{prop}

\begin{proof}
Let $f_1,f_2 \in \mc H$ with $f_1^2+f_2^2\in \mc H^{\times}$. 
We have $\H\setminus\bigcup\Max\H=\mg{\H}$, and since by the hypothesis $\Max\H$ is finite, we can apply
 \Cref{Value Zero} to choose $g_1\in\mg{\H}$ and $g_2\in \H$ such that $f_1^2+f_2^2=g_1^2+g_2^2=g_1^2(1+(g_1^{-1}g_2)^2)\in\sq{\H}\cdot(1+\H^2)$.
 This shows that $\mg\H\cap(\H^2+\H^2)  \subseteq \sq{\H}\cdot(1+\H^2)$.
 The opposite inclusion is obvious, because $1+\H^2 \subseteq \mg\H$.
\end{proof}

\begin{cor}\label{Semilocal square-effective}
Assume that $\H$ is a semilocal B\'ezout domain with $1+\H^2 \subseteq \mg\H$.
Then $\H$ is square-effective. 
\end{cor}

\begin{proof}
This follows by \Cref{Semilocal case} and \Cref{Square-effective k=2}.
\end{proof}

\section{Semilocal B\'ezout rings} 

Let $n\in\nat^+$ and let $F$ be a field of characteristic different from $2$.
In this section we present a technique to search for a subring $\H$ of $F$ satisfying the hypotheses of \Cref{Semilocal square-effective} and \Cref{C:square-effective-Mnotempty} and such that $p^\ast(\H)\leq 2^n$, so as to show that  $p(F)\leq 2^n+1$.
\medskip

By a \emph{B\'ezout ring of $F$} we mean a subring of $F$ which is a B\'ezout domain and whose fraction field is $F$.
Semilocal B\'ezout rings of $F$ are naturally related to valuation rings of $F$.

\begin{prop}\label{Intersection VRs}
Let $\mc H$ be a subring of $F$ with fraction field $F$.
Then the following are equivalent:
\begin{enumerate}[$(i)$]
\item $\mc{H}$ is a semilocal B\'ezout ring of $F$.
\item  $\mc{H}$ is a finite intersection of valuation rings of $F$.
\end{enumerate}
\end{prop}

\begin{proof}
\follows{i}{ii} 
Assume that $\H$ is a B\'ezout ring. Then for each $\mfm\in\Max\H$,  $\H_\mfm$ is a valuation ring of $F$, by \cite[Theorem 64]{Kap74}. Furthermore $\H=\bigcap_{\mfm\in \Max\H}\H_\mfm$, by \cite[Theorem 53]{Kap74}, and if $\H$ is semilocal, then this is a finite intersection.

\follows{ii}{i} See \cite[Theorem 107]{Kap74}. 
\end{proof}

We call a semilocal B\'ezout domain $\H$ \emph{square-effective of order $n$} if 
$$1+\sos{2^n-1}{\H} \subseteq \mg\H\quad\mbox{ and }\quad(1+\sos{2^n}{\H})\cap \Jac(\H)\neq \emptyset\,.$$
By \Cref{Semilocal square-effective}, the first of the two conditions implies  that $\H$ is square-effective. 
We can reformulate each of the two conditions in terms of levels of residue fields.

\begin{prop}\label{Levels of residues}
Let $\mc H$ be a domain and $k\in\nat^+$.
Then:
\begin{enumerate}[$(a)$]
\item $1+\sos{k-1}{\H} \subseteq \mg\H$ if and only if $s(\H/\mfm)\geq k$ for every $\mfm \in \Max\H$.
\item If $(1+\sos{k}{\H})\cap \Jac(\H)\neq \emptyset$, then $s(\H/\mfm)\leq k$ for every $\mfm \in \Max\H$.
\item If $\H$ is semilocal and such that $s(\H/\mfm)\leq k$ for every $\mfm \in \Max\H$, then $(1+\sos{k}{\H})\cap \Jac(\H)\neq \emptyset$.
\end{enumerate}
\end{prop}

\begin{proof}
Parts $(a)$ and $(b)$ follow immediately from the definition of the level and from the general facts that $\mg{\H}=\H\setminus \bigcup\Max{\H}$ and $\Jac(\H)=\bigcap\Max{\H}$.

$(c)$ Assume that $\H$ is semilocal with $s(\H/\mfm)\leq k$ for every $\mfm \in \Max\H$.
For $\mfm \in \Max\H$, we choose $f_{\mfm,1},\dots, f_{\mfm,k} \in \H$ with $ 1+f_{\mfm,1}^2+ \dots + f_{\mfm,k}^2\in \mfm$.
By the Chinese Remainder Theorem,  for $1\leq i\leq k$ we  find $f_i \in \H$ such that $f_i\equiv f_{\mfm,i} \bmod \mfm$ for all $\mfm\in\Max\H$.
Then $1+f_1^2+\dots +f_k^2\in \bigcap \Max\H=\Jac(\H)$. 
Hence $(1+\sos{k}{\H})\cap \Jac(\H)\neq \emptyset$.
\end{proof}

We say that $F$ is \emph{$n$-effective} if 
$F$ has a semilocal B\'ezout ring $\H$ which is square-effective of order $n$ and such that $p^*(\H)\leq 2^n$.

\begin{ex}\label{EX:s=p-n-effective}
Note that $F$ is square-effective of order $n$ if and only if $s(F)=2^n$.
Therefore, if we have that $s(F)=p(F)=2^n$, then $F$ is trivially $n$-effective.
\end{ex}

\begin{prop}\label{Bound for k-effective}
Let $n\in \nat^+$. 
If $F$ is $n$-effective, then $2^n\leq p(F)\leq 2^n+1$.
\end{prop}
\begin{proof}
Let $\H$ be a semilocal B\'ezout ring of $F$ which is square-effective of order $n$ and such that $p^*(\H)\leq 2^n$.
Then we have $1+\H^2\subseteq 1+\sos{2^n-1}{\H}\subseteq\mg{\H}$ and $(1+\sos{2^n}{\H})\cap \Jac(\H)\neq \emptyset$.
Hence $p(F)\leq p^\ast(\H)+1\leq 2^n+1$, by \Cref{C:square-effective-Mnotempty}. 

If we had $\sos{2^n}\H\subseteq\sos{2^n-1}\H$, then $1+\sos{2^n}\H\subseteq 1+\sos{2^n-1}{\H}\subseteq \mg \H$, whence $(1+\sos{2^n}{\H})\cap \Jac(\H)=\emptyset$.
Therefore there exists $h\in \sos{2^n}\H\setminus \sos{2^n-1}\H $.
By \Cref{Minimal Value Representation} $(d)$, we have
$\H\cap \sos{2^n-1}F=\sos{2^n-1}\H$. 
Hence $h\notin\sos{2^n-1}F$, whereby $p(F)\geq 2^n$.
\end{proof}

\begin{cor}
Let $n\in\nat^+$. If $F$ is nonreal and $n$-effective, then $s(F)=2^n$.
\end{cor}
\begin{proof}
Assume that $F$ is nonreal and $n$-effective.
Then $s(F)\leq p(F)\leq s(F)+1$, and by \Cref{Bound for k-effective}, we have $2^n\leq p(F)\leq 2^n+1$. 
Since $s(F)$ is a $2$-power and $n\geq 1$, the statement follows.
\end{proof}

Let $k \in \nat^+$ and let $\H$ be a subring of $F$.
We say that \emph{$\H$ characterises sums of $k$ squares in $F$} if 
$$\left\{h\in\mg{\H}\cap\sums{F}\mid h+\mfm\in\sos{k}{(\H/\mfm)}\mbox{ for any } \mfm\in \Max\H \right\}\subseteq \sos{k}{F}\,.$$
This condition allows one to check whether a unit in $\H$ which is a sum of squares in $F$ is a sum of $k$ squares in $F$ by inspecting residues modulo maximal ideals.

\begin{ex}\label{characterise and p^*}
Let $\H$ be a subring of $F$ with fraction field $F$.
Then for any $k\in\nat^+$ with $p^\ast(\H)\leq k$, we have that $\H$ characterises sums of $k$ squares in $F$.
\end{ex}

\begin{prop}\label{Inductive step}
Let $n\in\nat^+$. Let $R$ be a semilocal B\'ezout ring of $F$ which characterises sums of $2^n$ squares in $F$.
For every $\mfm \in \Max{R}$ let $\H(\mfm)$ be a semilocal B\'ezout ring of $R/\mfm$ which is square-effective of order $n$ and such that $p^*(\H(\mfm))\leq 2^n$.
Then 
$$\H=\{f\in R \mid f+ \mfm \in \H(\mfm) \mbox{ for all } \mfm \in \Max R\}$$
is a semilocal B\'ezout ring of $F$ which is square-effective of order $n$ and such that $p^*(\H)\leq 2^n$.
\end{prop}

\begin{proof}
Recall that
$$R=\bigcap_{\mfm\in\Max R} R_\mfm.$$ 

Fix $\mfm \in \Max{R}$. 
It follows by \cite[Theorem 107]{Kap74} that $R_\mfm$ is a valuation ring of $F$ with residue field $R/\mfm$. 
We denote by $\pi_\mfm: R_\mfm\rightarrow R/\mfm$ the residue homomorphism of $R_\mfm$.
Similarly, we have 
$$\H(\mfm)=\bigcap_{\mfm'\in\Max{\H(\mfm)}} {\H(\mfm)}_{\mfm'}\,.$$

Fix now
$\mfm'\in \Max{\H(\mfm)}$. 
It follows again by \cite[Theorem 107]{Kap74} that $(\H(\mfm))_{\mfm'}$ is a valuation ring of $R/\mfm$ with residue field $\H(\mfm)/\mfm'$. 
As $\H(\mfm)$ is square-effective of order $n$, we have $s(\H(\mfm)/\mfm')=2^n$, by \Cref{Levels of residues}.
For any valuation ring $\mc{O}$ of $R/\mfm$, we have that $\pi_\mfm^{-1}(\mc{O})$ is a valuation ring of $F$ having the same residue field as $\mc O$; see \cite[p.~45]{EP05}.
Thus $\pi_\mfm^{-1}(\H(\mfm))$ is the intersection of finitely many valuation rings of $F$ whose residue fields have level $2^n$. 

Since $R$ is semilocal, we obtain 
that
$\H$ is the intersection of finitely many valuation rings of $F$ with residue field of level $2^n$.
Thus $\H$ is a semilocal B\'ezout ring of $F$, by \Cref{Intersection VRs}, and it is square-effective of order $n$, by \Cref{Levels of residues}.

Let $f \in \mc{\mg H} \cap \sums{F}$.
For every $\mfm\in \Max{R}$, it follows by \Cref{Minimal Value Representation} $(e)$ that $f+\mfm\in\sums{(R/\mfm)}$.
By definition of $\H$ we have that $f+\mfm\in\mg{\mc{H}(\mfm)}\cap\sums{(R/\mfm)}$, and since $p^*(\mc{H}(\mfm))\leq 2^n$ we obtain that
$f+\mfm\in\sos{2^n}{(R/\mfm)}$.
Since $R$ characterises sums of $2^n$ squares in $F$, we conclude that $f\in \sos{2^n}{F}$. 
Thus $\mc{\mg{H}}\cap\sums{F}\subseteq \sos{2^n}{F}$, whereby $p^*(\mc{H})\leq 2^n$.
\end{proof}

Let $v$ be a valuation on $F$. We denote by $\mc O_v$ the valuation ring of $F$ associated to $v$, by $\mfm_v$ the maximal ideal of $\mc O_v$, by $Fv$ the residue field $\mc O_v/\mfm_v$, by $\pi_v$ the residue homomorphism $\mc{O}_v\to Fv$ and by $vF$ the value group $v(\mg{F})$.
For any subfield $E\subseteq F$, we write $Ev$ for the residue field of the restriction of $v$ to $E$, that is,  $Ev=\{x+(\mfm_v\cap E)\mid x\in\mc{O}_v\cap E\}\subseteq F v$.

Let $V$ be a set of valuations on $F$. 
We define $$\H_V=\bigcap\limits_{v\in V}\mc O_v\,.$$
For $k\in\nat^+$, 
we say that \emph{$V$ characterises sums of $k$ squares in $F$} if the ring $\H_V$ characterises sums of $k$ squares in $F$.

\begin{thm}\label{T:example0}
Let $n\in\nat^+$.
If there exists a nonempty finite set $V$ of valuations on $F$ that characterises sums of $2^n$ squares in $F$ and such that $Fv$ is $n$-effective for every $v\in V$, then $F$ is $n$-effective.
\end{thm}

\begin{proof}
Assume that $V$ is such a set of valuations on $F$.
Then $\H_V$ is a semilocal B\'ezout ring of $F$ that characterises sums of $2^n$ squares in $F$. 
For every $v\in V$ we may fix a semilocal B\'ezout ring $\H_v$ of $Fv$ which is square-effective of order $n$ and such that $p^\ast(\H_v)\leq 2^n$.
Set $R=\H_V$.
For every $\mfm \in \Max{R}$, we have by \cite[Theorem 107]{Kap74} that $\mfm = \mfm_v\cap R$ for some $v\in V$, and we set $\H(\mfm)=\H_v$ for such a $v\in V$.
In this setting, it follows by \Cref{Inductive step}
that the subring $\H$ constructed there is square-effective of order $n$ and that
$p^\ast(\H)\leq 2^n$. 
Hence $F$ is $n$-effective.
\end{proof}

A valuation with value group $\zz$ is called a \emph{$\zz$-valuation}.
We denote by $\Omega_F$ the set of $\zz$-valuations on $F$.
For $v\in\Omega_F$, we denote by $F^v$ the completion of $F$ with respect to $v$, which is given by the fraction field of the completion of the $\mfm_v$-adic completion of  $\mc{O}_v$.

\begin{prop}\label{characterises-sos-k}
Let $k\in\nat$ with $k\geq 2$ and let $V$ be a finite set of $\zz$-valuations on $F$ such that $s(Fv)\geq k$ for all $v\in V$.
Then $V$ characterises sums of $k$ squares if and only if 
$\sos{k}{F}=\sums{F}\cap \bigcap_{v\in V} \sos{k}{F^v}$.
\end{prop}

\begin{proof}
Set $S=\{x \in \mg{\H}_V \cap \sums F \mid x+\mfm_v \in  \sos{k}{(Fv)} \, \text{ for every } \, v \in V\}$ and $$T=\sums{F}\cap \bigcap_{v\in V} \sos{k}{F^v}.$$
Note that $\sos{k}{F}\subseteq T$.
Hence we have to show that $S\subseteq \sos{k}{F}$ if and only if $T\subseteq \sos{k}{F}$.
We claim that $T=F^2 \cdot S$. 
The statement then follows trivially from this equality.

Consider $v\in V$. 
Recall that $v$ extends uniquely to a $\zz$-valuation $v'$ on $F^v$ with $F^vv'=Fv$.
We define  $S_v=\{x \in \mg{\mc O}_{v'} \cap \sums{F^v} \mid x+\mfm_{v'} \in  \sos{k}{(Fv)}\}$.
Since $s(Fv)\geq k\geq 2$, we obtain by \Cref{Levels of residues} that $1+\sos{k-1}{\mc O_{v'}} \subseteq \mg{\mc{O}}_{v'}$ and that $v'(2)=0$.
It follows by \Cref{Minimal Value Representation}~$(c)$ and $(e)$ that
$\sos{k}{F^v}\subseteq F^{v 2} S_v$.
As $F^v$ is henselian with respect to $v'$ and 
$v'(2)=0$, we also have the converse inclusion. 
Hence $\sos{k}{F^v}=F^{v2}  S_v$.
This implies that $F^2 (\sums F\cap S_v)\subseteq\sums F\cap\sos{k}{F^v}$.
We claim that the opposite inclusion also holds.
Consider $x\in (\sums F\cap\sos{k}{F^v})\setminus \{0\}$.
Using that $v'F^v=vF$, $F^vv'=Fv$ and $\sos{k}{F^v}\subseteq F^{v 2} S_v$, we obtain that 
there exists $c\in \mg{F}$ such that $c^2x\in\mg{\mc{O}}_v\cap\sos{k}{F^v}\subseteq S_v$, whereby $x\in F^2 (\sums F\cap S_v)$.
Therefore we have
$\sums F\cap\sos{k}{F^v}= F^2 (\sums F\cap S_v)$.
Hence
$$T= \bigcap_{v\in V} (\sums{F}\cap\sos{k}{F^v})=\bigcap_{v\in V} F^2 (\sums F\cap S_v).$$
Observe that $S=\sums{F}\cap \bigcap_{v\in V} S_v \subseteq \bigcap_{v\in V} F^2 (\sums F\cap S_v)=T$.
Hence it remains to be shown that $T\subseteq F^2\cdot S$.
To this purpose, consider $f\in T\setminus \{0\}$. 
For every $v\in V$ we fix $g_v \in F$ and $h_v\in S_v$ 
such that $f=g_v^2h_v$. 
Since $V$ is a finite set of $\zz$-valuations on $F$, which are necessarily pairwise independent, we can apply the Approximation Theorem \cite[Theorem~2.4.1]{EP05} to obtain 
an element $g\in \mg F$ such that $v(g-g_v)> v(g_v)$ for every $v \in V$.
It follows that $f/g^2\in \bigcap_{v\in V}S_v$.
Hence
$f\in F^2 (\sums{F}\cap \bigcap_{v\in V} S_v)=F^2\cdot S$.
\end{proof}

By a \emph{quadratic form} we mean a homogeneous polynomial of degree $2$.
We say that a class $\mc{C}$ of quadratic forms over $F$ \emph{satisfies the local-global principle with respect to $\Omega_F$} if, for every form $q\in\mc{C}$ that has a nontrivial solution over $F^v$ for every $v\in\Omega_F$, $q$ has a nontrivial solution over $F$.

\begin{cor}\label{C:characterise-criterion}
Let $k\in\nat$ with $k\geq 2$. 
Suppose that quadratic forms over $F$ of the shape $\sum_{i=1}^k X_i^2 - aX_0^2$ with $a\in \mg{F}$ over $F$ satisfy the local-global principle with respect to $\Omega_F$ and that the set $V=\{v\in\Omega_F\mid p(F^v)>k\}$ is finite and contains no dyadic valuation. 
Then $V$ characterises sums of $k$ squares in $F$.
\end{cor}
 \begin{proof}
 It follows from the hypotheses that 
  $$\sos{k}{F}=\sums{F}\cap \bigcap_{v\in \Omega_F} \sos{k}{F^v}=\sums{F}\cap \bigcap_{v\in V} \sos{k}{F^v}\,.$$
For any $v\in V$, we have $k< p(F^v)\leq s(Fv)+1$, whereby $s(Fv)\geq k$. 
Hence the statement follows by \Cref{characterises-sos-k}.
 \end{proof}

\section{Function fields in one variable} 

The tools developed in the previous sections can be used to compute the Pythagoras numbers of certain function fields in one variable over $\ffps{k}{t}$ for some base fields $k$.
Note that any field $k$ is relatively algebraically closed in $\ffps{k}{t}$, and in particular any irreducible polynomial in $k[X]$ remains irreducible in $\ffps{k}{t}[X]$.

\begin{prop}\label{Hyperelliptic}
Let $n\in\nat$.
Let $k$ be a real field such that $p(L)\leq 2^n$ for every finite field extension $L/k$.  
Let $h\in k[X]$ be irreducible and such that $h \in\sums{k(X)}\setminus \sos{2^{n+1}-1}{k(X)}$.
Then $$p\left(k(\!(t)\!)(X)\big(\sqrt{(tX-1)h}\big)\right)=2^{n+1}+1\,.$$
\end{prop}

\begin{proof}
We set $K=k(\!(t)\!)$, $f=(tX-1)h \in K[X]$ and $F=K(X)(\sqrt{f})$.
We further set $K'=K(\vartheta)$ and $k'=k(\vartheta)$ for some root $\vartheta$ of $h$.
Observe that $K'/K$ and $k'/k$ are finite extensions, and that the $t$-adic valuation on $K$ extends to a $\zz$-valuation on $K'$ with residue field $k'$, and hence $K'$ can be identified with $k'(\!(t)\!)$.
In particular, we have that $\mg{K'}=(\mg{k'}\cup t\mg{k'})\cdot\sq{K'}$.
Since  $p(k')\leq 2^n$, it follows that $\sums{K'}\subseteq\sq{K'}(\sos{2^n}{K'}\cup t\sos{2^n}{K'})$.
In particular $\lvert \mg{(\sums{K'})}/\mg{(\sos{2^{n}}{K'})}\rvert \leq 2$.
We set $m=n+1$ and obtain by \cite[Theorem 3.10]{BVG09} that $\lvert \mg{(\sums F)}/\mg{(\sos{2^{m}}{F})}\rvert \leq 2$.
By \cite[Theorem 3.5]{BVG09}, the hypotheses further implies that $p(k(X))\leq 2^{m}$ and consequently $h\in \sos{2^m}{k(X)}\setminus \sos{2^m-1}{k(X)}$.

We consider the Gauss extension to $K(X)$ of the $t$-adic valuation on $K$ with respect to the variable $X$ (see e.g. \cite[Corollary 2.2.2]{EP05}), and we denote by $v$ an extension of this valuation from $K(X)$ to $F$.
Note that $vK(X)=vK=\zz$ and that $K(X)v$ can be naturally identified with $k(X)$.
The residue of $f$ with respect to $v$ is given by $-h$.
Hence $Fv=k(X)(\sqrt{-h})$.
As $-h$ is not a square in $k(X)$, we have $[Fv:K(X)v]=2=[F:K(X)]$.
By \cite[Theorem 3.3.4]{EP05}, it follows that $vF=vK(X)=vK=\zz$.
Since $k(X)$ is real and $h\in \sos{2^m}{k(X)}\setminus \sos{2^m-1}{k(X)}$, we obtain by \cite[Theorem 4.4.3~$(i)$]{Scha85} that $s(Fv)= 2^m$.
Since $vF=\zz$, it follows by \cite[Proposition 4.1]{BGVG14} that $v(\sos{2^{m}}{F})\subseteq 2\zz$.
As $tX\in\sums{F}$ and $v(tX)=1$,
we obtain that $p(F)>2^m$ and $\mg{\mc{O}}_v\cap tX\sos{2^m}{F}=\emptyset$.
Since $|\sums{F}/\sos{2^m}{F}|\leq 2$, we conclude that $\sums{F}=\sos{2^m}{F}\cup tX\sos{2^m}{F}$ and $\sums{F}\cap \mg{\mc{O}}_v\subseteq \sos{2^n}{F}$.
Hence $\mc{O}_v$ is a B\'ezout ring of $F$ with $p^*(\mc{O}_v)\leq 2^m$.

Since $s(Fv)= 2^m$, it follows by \Cref{Levels of residues} that $\mc{O}_v$ is square-effective of order $m$. 
Hence $F$ is $m$-effective, whereby $p(F)\leq 2^m+1$, in view of \Cref{Bound for k-effective}.
This proves that $p(F)=2^{m}+1$.
\end{proof}

\begin{ex}\label{Tikhonov curve}
Let $F$ denote the function field of the elliptic curve  $$Y^2=(tX-1)(X^2+1)$$ over $\rr(\!(t)\!)$.
Applying \Cref{Hyperelliptic} to $k=\rr$ and $h=X^2+1$, we obtain a new argument that 
$$p(F)=3\,.$$

The observation that $3\leq p(F)\leq 4$ goes back to \cite[Example 3.10]{TVGY06}, and the equality $p(F)=3$ was obtained in \cite[Corollary 6.13]{BGVG14}, by a less elementary method.
\end{ex}

\begin{ex}\label{Motzkin curve}
As mentioned in the introduction, it was shown in \cite{CEP71} that 
the Motzkin polynomial $M(X,Y)=X^2Y^4+X^4Y^2-3X^2Y^2+1$ is a sum of $4$ squares but not a sum of $3$ squares in $\rr(X,Y)$.
We consider the field $K=\rr (Y)(\!(t)\!)$.
Note that $M(X,Y)$ is irreducible in $\rr(Y)[X]$, and hence also in $K[X]$, and that 
$M(X,Y)\in\sos{4}{K(X)}\setminus\sos{3}{K(X)}$.
We now consider the field $$F=K(X)\left(\sqrt{(tX-1)M(X,Y)}\right)\,.$$
By \Cref{Hyperelliptic}, we obtain that $p(F)=5$.
\end{ex}

Note that \Cref{Tikhonov curve} and \Cref{Motzkin curve} imply that the bounds in \Cref{P:fufi-Pn-Pythagoras} and \Cref{Complete base field} below are sharp for $n=1$ and $n=2$.
After these examples of particular function fields in one variable, we now turn to consider 
base fields where our method gives us a bound on the Pythagoras numbers of all function fields in one variable.

Recall that by a \emph{function field in one variable} we mean a finitely generated field extension of transcendence degree $1$. 

Let $n \in \nat$ and let $K$ be a field. 
We call $K$ a \emph{$\mathcal{P}_n$-field} if $p(K(X))\leq 2^{n+1}$ and if every function field in one variable $F/K$ with $p(F)>2^{n+1}$ is $(n+1)$-effective.

\begin{ex}\label{FPS}
If $p(E)\leq 2^{n+1}$ holds for every function field in one variable $E/K$,
then $K$ is a $\mc P_n$-field.
This applies in particular to the following situations:
\begin{enumerate}[$(i)$]
    \item For any $n\geq 0$ to $K=\rr(X_1,\dots,X_{n})$, by \cite[Theorem 1]{Pfi67}; see also \cite[Theorem XI.4.10]{Lam05}. For $n=0$ this result goes back to \cite{Witt34}.
    \item More generally, if $K(\sqrt{-1})$ is a $\mc{C}_n$-field, in terms of Tsen-Lang theory (see e.g.~\cite[\S 2.15]{Scha85}).
    Indeed, this implies for any function field in one variable $F/K$ that $F(\sqrt{-1})$ is a $\mc{C}_{n+1}$-field. 
    In particular, every $(n+1)$-fold Pfister form over $F$ represents all elements of $F(\sqrt{-1})$, whereby \cite[Corollary XI.4.9]{Lam05} yields that $p(F)\leq 2^{n+1}$.
    \item For $n\geq 2$ to $K=\qq(X_1,\dots,X_{n-1})$.
    This follows from \cite[Theorem 4.1.2~$(c)$]{CTJ91}, which 
    relies on two deep facts \cite[Conjectures 2.1 and 2.5]{CTJ91} 
 proven later in \cite[Theorem 0.1]{Ja16} and \cite[Theorem 4.1]{OVV07}.
 \end{enumerate}
\end{ex}

The interest of the notion of $\mc{P}_n$-field lies in the following consequence.

\begin{prop}\label{P:fufi-Pn-Pythagoras}
Let $n\in\nat$ and let $K$ be a $\mc{P}_n$-field.
Then $p(F)\leq 2^{n+1}+1$ for every function field in one variable $F/K$.
\end{prop}
\begin{proof}
By \Cref{Bound for k-effective}, this follows from the definition of $\mc P_n$-field.
\end{proof}

We will now show for $n \in \nat$ that the class of $\mc{P}_n$-fields is stable under passage from a field $K$ to the power series field $K(\!(t)\!)$.
To show this, we will use the methods developed in the previous sections.

A function field in one variable $F/K$ is called \emph{ruled} if there exist $\theta \in F$ and a finite field extension $K'/K$ such that $F=K'(\theta)$, and \emph{nonruled} otherwise.

\begin{prop}\label{T:example1}
Let $m \in \nat^+$. 
Let $K$ be a field such that $p(K(X))\leq 2^m$. 
Let $F/K(\!(t)\!)$ be a function field in one variable and let $v\in\Omega_F$. 
Then $K\subseteq \mc{O}_v$, and 
if $p(F^v)>2^m$, then $Fv/K$ is a nonruled function field in one variable.
\end{prop}

\begin{proof}
If $K(\!(t)\!)\subseteq\mc{O}_v$, then
$p(F^v)\leq p(K(X))\leq 2^m$ by \cite[Lemma 6.3]{BGVG14}.
Assume thus that $K(\!(t)\!)\not\subseteq\mc{O}_v$. 
Since $v\in \Omega_F$, we obtain by \cite[Proposition 2.2]{BGVG14} that $\mc O_{v}\cap K(\!(t)\!)=K[\![t]\!]$.
Hence $K\subseteq Fv$ and $K(\!(t)\!)v=K$.
Since $p(K(X))\leq 2^m$, it follows by \cite[Corollary 4.13]{BGVG14} that $p(L(X)) \leq 2^m$ holds for every finite field extension $L/K$, and hence also for every algebraic extension $L/K$.

Assume that $Fv/K$ is algebraic. 
Then $p(Fv(X))\leq 2^m$, and we conclude by using \cite[Theorem 4.14]{BGVG14} that $ p(F^v(X))\leq 2^m$. 
Thus $p(F^v)\leq 2^m$.

Consider now the case where $Fv/K$ is transcendental.
Then $Fv/K$ is a function field in one variable, by \cite[Theorem 3.4.3]{EP05}.
Suppose that $Fv/K$ is ruled.
Then $Fv=L(\theta)$ for a finite field extension $L/K$ and some element $\theta\in Fv$ which is transcendental over $K$.
In order to show that $p(F^v)\leq 2^m$, we may assume that $\car K\neq 2$.
Since $p(K(X))\leq 2^m$, it then follows by \cite[Corollary 4.13]{BGVG14} that $p(Fv)=p(L(X))\leq 2^m$. 
If $Fv$ is real, then $p(F^v)=p(Fv)\leq 2^m$. 
If $Fv$ is nonreal, then $s(F^v)=s(Fv)< 2^m$ by \cite[Theorem 3.5]{BVG09}, and therefore $p(F^v)=s(F^v)+1\leq 2^m$. 
\end{proof}

\begin{thm}\label{T:example2}
Let $n \in \nat$.
If $K$ is a $\mathcal{P}_n$-field, 
then $K(\!(t)\!)$ is a $\mc P_n$-field.
\end{thm}

\begin{proof}
Let $K$ be a $\mc{P}_n$-field.
If $\car K=2$, then $\ffps{K}{t}$ is trivially a $\mc P_n$-field.
Assume now that $\car K\neq 2$.
Since $p(K(X))\leq 2^{n+1}$, we get by \cite[Theorem 4.14]{BGVG14} that $p(K(\!(t)\!)(X))\leq 2^{n+1}$.
Consider a function field in one variable $F/K(\!(t)\!)$ with $p(F)>2^{n+1}$.
Let $V=\{v\in\Omega_F\mid p(F^v)>2^{n+1}\}$ and let  $$W=\{w\in\Omega_F\mid \mbox{$Fw/K$ nonruled function field in one variable}\}\,.$$ 
Then $V\subseteq W$, by \Cref{T:example1}, and $W$ is finite, by \cite[Corollary 3.9]{BGVG14}. Therefore $V$ is finite.
By \cite[Theorem 3.1]{CTPS12}, quadratic forms in at least $3$ variables over $F$ satisfy the local-global principle with respect to $\Omega_F$.
Since $2^{n+1}+1\geq 3$, we conclude by \Cref{C:characterise-criterion} that $V$ characterises sums of $2^{n+1}$ squares in $F$.

In particular, since $p(F)>2^{n+1}$, we obtain that $V\neq \emptyset$.
Consider any $v\in V$. Then $p(F^v)>2^{n+1}$,
 and hence either $s(Fv)=2^{n+1}$ or $p(Fv)>2^{n+1}$.
In view of \Cref{EX:s=p-n-effective}, and because $Fv/K$ is a function field in one variable and $K$ is a $\mc{P}_n$-field, we obtain in either case that $Fv$ is $(n+1)$-effective.
Now, given that $Fv$ is $(n+1)$-effective for every $v \in V$, we conclude by \Cref{T:example0} that
$F$ is $(n+1)$-effective. 

This argument shows that $K(\!(t)\!)$ is a $\mc P_n$-field.
\end{proof}

We retrieve the following statement contained in \cite[Theorem 6.13]{BGVG14}.

\begin{cor}\label{Complete base field}
Let $n,r\in\nat$.
Let $K$ be a field such that $p(E)\leq 2^{n+1}$ for every function field in one variable $E/K$.
Let $F$ be a function field in one variable over $K(\!(t_1)\!)\dots(\!(t_r)\!)$.
Then
$p(F)\leq 2^{n+1}+1$.
\end{cor}

\begin{proof}
It follows by \Cref{FPS}, via an iterated application of \Cref{T:example2}, that $K(\!(t_1)\!)\dots(\!(t_r)\!)$ is a $\mc P_n$-field. 
Hence, either $p(F)\leq 2^{n+1}$ or $F$ is $(n+1)$-effective.
In view of \Cref{Bound for k-effective}, we conclude that $p(F)\leq 2^{n+1}+1$.
\end{proof}

\section{Geometric global fields}\label{dim 2} 

It is shown in \cite{Hu15} that $p(F)\leq 3$ for every finite field extension $F$ of  $\ffps{\rr}{t_1,t_2}$. 
In this section, we show that one can also obtain this bound by means of \Cref{Bound for k-effective}.
To this aim, we study the discrete valuations on a finite field extension of the fraction field of a complete noetherian local domain.

Let $R$ be a commutative ring. By the \emph{dimension of $R$} we refer to its Krull dimension and we denote it by $\dim R$.
Assume now that $R$ is a local ring. We denote by $\mfm_R$ its unique maximal ideal and by $\kappa_R$ the residue field $R/\mfm_R$.
Recall that $R$ is called \emph{henselian} if for every monic polynomial $f\in R[X]$, any simple root of $f$ in $\kappa_R$ is the residue of a root of $f$ in $R$.
By \cite[Lemma 10.153.9]{Sta} 
every complete local domain is henselian.

We will mostly focus on the case where $R$ is $2$-dimensional.
We will show that in this case the fraction field of $R$ has only finitely many discrete valuation rings whose residue field is a nonruled function field in one variable over $\kappa_R$.

\begin{prop}\label{Cohen structure}
Let $d \in \nat^+$ and let $R$ be a complete noetherian local domain with $\dim R=d$. 
Then there exist subrings $\mc{O}\subseteq R_0\subseteq R$ such that $\mc{O}$ is
a complete discrete valuation ring with residue field  $\kappa_R$, $R_0 \simeq \rfps{\mc O}{t_1,\dots,t_{d-1}}$ and $R$ is a finite $R_0$-algebra.
\end{prop}

\begin{proof}
By \cite[Lemma 10.161.11]{Sta}, there exists a complete regular local domain $R_0$ such that $R$ is a finite $R_0$-algebra and such that $R_0$ is either isomorphic  to $\kappa_{R}[\![t_1,\dots, t_d]\!]$ or to $\mc O[\![t_1,\dots, t_{d-1}]\!]$ for  a complete discrete valuation ring $\mc O$ with residue field $\kappa_R$.
In the first case, we let $\mc O=\rfps{\kappa_R}{t_d}$, which is then a complete discrete valuation ring with residue field $\kappa_R$.
\end{proof}

\begin{lem}\label{centred}
Let $F$ be a field, $R$ a henselian local subring of $F$ and $\mc O$ a discrete valuation ring of $F$ such that $R=\mfm_R+R\cap \mc O$.
Then $R\subseteq \mc O$.
\end{lem}

\begin{proof}
Since $R$ is henselian, we have $1+\mfm_R \subseteq F ^{\times n}$ for any $n \in \nat$ coprime to $\car{\kappa_R}$. 
Since this holds for infinitely many natural numbers $n$ and $\mc{O}$ is a discrete valuation ring, we conclude that $1+\mfm_R\subseteq \mg{\mc O}$. 
In particular $\mfm_R\subseteq \mc O$. Therefore $R=\mfm_R+R\cap \mc O\subseteq \mc O$.
\end{proof}

Let $F$ be a field.
We denote by $\Omega_{F}$ the set of $\zz$-valuations on $F$. 
Given $v\in \Omega_F$ and a subring $R\subseteq F$, we say that $v$ \emph{is centred on $R$} if $R\subseteq \mc O_v$; in this case, for $\mf p \in \mathsf{Spec}(R)$, we say that $v$ \emph{is centred on $R$ in $\mf p$} if $\mfm_v \cap R= \mf p$. 

Assume now that $F$ is the function field of an integral scheme $\mc X$. 
For $x \in \mc X$, we denote by $\mc O_{\mc X, x}$ the stalk of $\mc X$ at $x$, by $\mfm_x$ its maximal ideal, and we set $\kappa(x)=\mc O_{\mc X, x}/\mfm_ x$.
For $v \in \Omega_F$, we say that $v$ \emph{is centred on $\mc X$ in $x$} if $\mc O_{\mc X, x}\subseteq \mc O_v$ and $\mfm_v \cap \mc O _{\mc X, x}= \mfm_ x$. 
\smallskip

In the sequel let $R$ be a complete regular local domain that is not a field.
We denote by $E$ the fraction field of $R$ and we consider a finite field extension $F/E$.

\begin{prop}\label{base ring}\label{P:center-nonzeroprime}
Every $\zz$-valuation on $F$ is centered on $R$ in a nonzero prime ideal of $R$.
\end{prop}
\begin{proof}
Set $d=\dim R$. 
Since $d>0$, it follows by \Cref{Cohen structure} that there exist subrings 
$\mc{O}\subseteq R_0\subseteq R$ such that $R$ is a finite $R_0$-algebra, $\mc{O}$ is a complete discrete valuation ring and $R_0\simeq \mc{O}[\![t_1,\dots,t_{d-1}]\!]$.
Let $K\subseteq F$ be the fraction field of $\mc{O}$.
Since $\mc O$ is complete, we have that $\mc O$ is the unique discrete valuation ring of $K$; see e.g.~ \cite[Proposition 2.2]{BGVG14}. 
Let $v\in\Omega_F$. 
We have that $\mc O_{v}\cap K=K$ or $\mc O_{v}\cap K=\mc O$, so in
any case $\mc{O}\subseteq \mc{O}_v$.
Therefore $$R_0=\mfm_{R_0}+\mc{O}\subseteq \mfm_{R_0}+(R_0\cap \mc O_v).$$ 
Since $R_0$ is complete and in particular henselian, we obtain by \Cref{centred} that $R_0\subseteq \mc O_v$. 
Since $R$ is a finite $R_0$-algebra, and thus an integral extension of $R_0$, we conclude that $R\subseteq \mc{O}_v$.
It follows that $\mfm_v \cap R  \in \Spec R$. 
Since $F/E$ is a finite field extension,
the restriction of $v$ to $E$ is nontrivial.
As $E$ is the fraction field of $R$, we conclude that  $\mfm_v \cap R \neq \{0\}$.
\end{proof}

In the following we will use some results from algebraic geometry, including resolution of singularities for surfaces \cite[p.~193]{Lip75}, for which we need the following observation.

\begin{rem}\label{excellent}
By \cite[Proposition 15.52.3]{Sta} a complete notherian local ring is excellent, and hence in particular universally catenary and a Nagata ring.
Here, this applies for $R$ as well as for the integral closure of $R$ in $F$, which is again a complete noetherian local ring by \cite[Theorem 4.3.4]{HS06}.
\end{rem}

\begin{prop}\label{residues}
Assume that $\dim(R)=2$. Let $v \in \Omega_{F}$ and let $\mfp=\mfm_v\cap R$.
Then one of the following holds:
\begin{enumerate}[$(i)$]
\item $\mf p$ is a principal ideal of height $1$ of $R$ and there exists a complete discrete valuation ring of $Fv$ whose residue field is a finite field extension of $\kappa_R$.
\item $\mf p =\mfm_R$ and $Fv$ is either an algebraic extension of $\kappa_R$ or a function field in one variable over $\kappa_R$.
\end{enumerate}
\end{prop}

\begin{proof}
In view of \Cref{P:center-nonzeroprime}, we have $\{0\}\subsetneq \mfp\subseteq \mfm_R$.
Since $R$ is a regular local ring, it is a unique factorization domain, by \cite[Theorem 4.2.16]{Liu06}.
Hence any height-$1$ prime ideal of $R$ is principal.

Assume first that $\mfp\neq \mfm_R$.
Since $R$ is local and $2$-dimensional, we obtain that $\mfp$ is a principal ideal of height $1$, $\mc O_v\cap E=R_{\mf p}$ and $R/\mf p$ is a $1$-dimensional complete local domain with residue field $\kappa_R$.
Since $\mc O_v\cap E=R_{\mf p}$, we have $E{v} \simeq \mathsf{Frac}(R/\mf p)$. 

By \Cref{Cohen structure}, there exists a complete discrete valuation ring $\mc O$ with residue field $\kappa_R$ which is a subring of $R/\mfp$ and such that $R/\mf p$ is a finite $\mc O$-algebra. 
Let $K$ be the fraction field of $\mc O$.
Then $E{v}/K$ is a finite extension.
Hence $Fv/K$ is a finite extension. 
It follows by \cite[Theorem 14:1]{OM73} that $Fv$ has a complete discrete valuation ring $\mc{O}'$ such that $\mc{O}'\cap K=\mc{O}$, and its residue field is a finite extension of the residue field of $\mc{O}$, which is $\kappa_R$.

Assume now that $\mf p=\mfm _R$. 
As $\dim R=2$, 
it follows by \cite[Theorem 8.3.26~$(a)$]{Liu06} that $\trdeg{\kappa_R}{Ev} \leq 1$.
Since the extension $F/E$ is finite, we obtain that $\trdeg{\kappa_R}{Fv} \leq 1$.
If $\trdeg{\kappa_R}{Fv} =0$, then $Fv/\kappa_R$ is algebraic.
Assume that $\trdeg{\kappa_R}{Fv} =1$. 
Let $S$ be the integral closure of $R$ in $F$.
In view of the properties of $S$ pointed out in  \Cref{excellent},
we obtain by \cite[Theorem 8.3.26~$(b)$]{Liu06} that there exists an integral scheme $\mc X$, a proper birational morphism $\mc X \rightarrow \Spec{S}$ and $x \in \mc X$ of codimension $1$ such that $\mc O_v=\mc O_{\mc X, x}$. 
Since any proper morphism is of finite type, there exists an open affine neighbourhood $C$ of $x$ in $\mc X$ such that $\mc O_{\mc X}(C)$ is a finitely generated $S$-algebra, and hence also a finitely generated $R$-algebra. Since $\mc O_v$ is a localisation of $\mc O_{\mc{X}}(C)$, it follows that the residue field extension
 $Fv/\kappa_R$ is finitely generated.
Hence $Fv/\kappa_R$ is a function field in one variable.
\end{proof}

Let $\mc X$ be a scheme. 
For $i \in \nat$, we denote by $\mc X^{(i)}$ the set of points of $\mc X$ of codimension $i$.
A scheme $\mc X'$ together with a birational proper morphism $\X'\rightarrow \mc X$ is called a \emph{model of $\mc X$}.

\begin{cor}\label{divisorial}
Assume that $\dim R=2$.
Let $S$ be the integral closure of $R$ in $F$ and let $v \in \Omega_{F}$. 
There exists a model $\mc X$ of $\Spec S$ on which $v$ is centred in a point of $\mc X ^{(1)}$ if and only if $Fv$ is either a function field in one variable over $\kappa_R$ or a complete discretely valued field whose residue field is a finite extension of $\kappa_R$.
\end{cor}

\begin{proof}
Set $\mc X_0=\Spec S$ and $\mf p=\mfm_v \cap S$.
Since $S$ is integral over $R$, the height of $\mfp$ in $S$ is equal to the height of $\mfp\cap R=\mfm_v \cap R$ in $R$, and hence it follows by \Cref{residues} that it is either $1$ or $2$.
Hence $\mf p \in \mc X_0^{(1)}$ or $\mf p \in \mc X_0^{(2)}$.

Assume first that $\mf p \in \mc X_0^{(1)}$. 
Then $\mc X_0$ itself is a model of $\Spec S$ such that $v$ is centred in a point of $\mc X_0^{(1)}$, and
by \Cref{residues}, $Fv$ is a complete discretely valued field whose residue field is a finite extension of $\kappa_R$.

Assume now that $\mf p \in \mc X_0^{(2)}$.
Recall that $E$ is the fraction field of $R$.
We consider the chain of field extensions $\kappa_R\subseteq \kappa(\mf p)\subseteq Ev \subseteq Fv$. 
Since $S$ is an integral extension of $R$, $\kappa(\mf p)/\kappa_R$ is an algebraic extension.
By \cite[Theorem 8.3.26~$(b)$]{Liu06}, there exists a model $\mc X$ of $\mc X_0$ such that $v$ is centred in a point of $\mc X ^{(1)}$ if and only if the extension $Fv/\kappa(\mf p)$ is transcendental, that is, if and only if $Fv/\kappa_R$ is transcendental.
Assume that we are in this case.
Since the extension $F/E$ is finite, so is $Fv/Ev$.
Hence $Ev/\kappa_R$ is transcendental.
It follows by \Cref{residues} that $Ev/\kappa_R$ is a function field in one variable.
As the extension $Fv/Ev$ is finite, we obtain that $Fv/\kappa_R$ is a function field in one variable.
\end{proof}

Let $\mc X$ be a scheme. 
Given $x \in \mc X$, we denote by $V(x)$ the Zariski closure of $\{x\}$ in $\mc X$, considered with its reduced scheme structure induced by $\mc X$.

\begin{prop}\label{ruled}
Assume that $\dim R=2$.
Let $S$ be the integral closure of $R$ in $F$ and let $\X$ 
be a regular model of $\Spec S$. 
Let $v \in \Omega_{F}$ and let $x$ be the centre of $v$ on $\X$.
Assume that $x \in \Xcod{2}$.
If $Fv/\kappa(x)$ is transcendental, then $Fv/\kappa_R$ is a ruled function field in one variable.
\end{prop}

\begin{proof}
In view of the properties of $S$ mentioned in \Cref{excellent} and since $Fv/\kappa(x)$ is transcendental, it follows by \cite[Theorem 8.3.26~$(b)$]{Liu06} that $$\trdeg{\kappa(x)}{Fv}= 1=\dim{\mc O_{\X,x}}-1.$$ 
We set $\X_0=\X$ and $x_0=x$. 
For $i\in \nat$ we define recursively $\pi_{i+1}: \mc X_{i+1}\rightarrow \mc X_i$ as the blowing up of $\mc X_{i}$ along the regular subscheme $V(x_{i})$ and denote by $x_{i+1}$ the centre of $v$ on $\mc X_{i+1}$.
By \cite[Theorem 8.1.19~$(a)$ and Proposition 8.1.12~$(b)$]{Liu06}, it follows for every $i \in \nat$ that $\mc X_{i+1}$ is a regular model of $\mc X_i$, and hence of $\Spec S$. 
Furthermore, it follows by \cite[Exercise 8.3.14]{Liu06} that there exists $n\in \nat^+$ such that $ x_n \in \mc X_n^{(1)}$, $\mc O_v=\mc O_{\mc X_n, x_n}$ and $x_i \in \mc X_i^{(2)}$ 
for $0\leq i < n$.
In particular, $Fv=\kappa(x_n)$ and, by \cite[Theorem 8.2.5]{Liu06},  $\kappa(x_i)/\kappa_R$ is a finite extension for every $0\leq i < n$. 
Since $x_n\in\smash{\mc X_n^{(1)}}$ and the exceptional fibre $\pi_{i+1}^{-1}{x_i}$ is an irreducible subscheme of $\mc X_{i+1}$ of dimension $1$, we obtain by \cite[Theorem 8.1.19~$(b)$]{Liu06} that $V(x_n) \simeq \mathbb{P}_{\kappa(x_{n-1})}^{1}$, and we conclude that $x_n$ is the generic point of $\pi_{i+1}^{-1}{x_i}$.
Therefore $\kappa(x_n)/\kappa(x_{n-1})$ is a rational function field in one variable. 
Since $Fv=\kappa(x_n)$ and $\kappa(x_{n-1})/\kappa_R$ is a finite extension, we conclude that $Fv/\kappa_R$ is a ruled function field in one variable.
\end{proof}

We obtain an analogue to \cite[Theorem 5.3]{BGVG14}.

\begin{prop}\label{fin val}
Assume that $\dim R=2$.
Then there exist only finitely many $\zz$-valuations on $F$ whose residue fields are nonruled function fields in one variable over $\kappa_R$.  
\end{prop}

\begin{proof}
Let $S$ be the integral closure of $R$ in $F$.
Then $S$ is excellent; see \Cref{excellent}.
Hence, by \cite[p. 193]{Lip75}, there exists a regular model $\eta: \mc X\rightarrow \Spec S$ of $\Spec S$. 
We denote by $\iota: \Spec S \rightarrow \Spec R$ the morphism of schemes corresponding to the inclusion $R\to S$. 
Denote by $\mc X_s$ the fibre of $\eta \circ \iota$ over $\mfm_R$. 
Since $\eta$ is birational, its image is dense in $\Spec S$, hence $\mc X_s$ has dimension at most $1$.
Since $\mc X_s$ is a closed subscheme of $\mc X$ of dimension at most $1$, it has only finitely many irreducible components, and hence we conclude that $\mc X_s \cap \mc X^{(1)}$ is finite.

Consider now an arbitrary $v \in \Omega_{F}$ such that $Fv/\kappa_R$ is a nonruled function field in one variable. We claim that $\mc{O}_v=\mc{O}_{\X,x}$ for some $x\in\mc X_s \cap \mc X^{(1)}$.
Since $\mc X_s \cap \mc X^{(1)}$ is finite, this will establish the statement.

Let $x\in \mc X$ be the centre of $v$ on $\mc X$. 
Since $Fv/\kappa_R$ is a function field in one variable, it follows by \Cref{residues} that $v$ is centred in $\mfm_R$ on $R$, that is, $x \in \mc X_s$.
Since $Fv/\kappa_R$ is a nonruled function field in one variable, it follows by \Cref{ruled} that $x \in \mc X^{(1)}$.
Since $\X$ is regular, it follows that $\mc O_{\mc X, x}$ is a discrete valuation ring of $F$.
Since $\mc O_{\X,x}\subseteq \mc O_v$, we obtain that $\mc O_{\X,x}= \mc O_v$.
\end{proof}

\begin{prop}\label{FPS in 2 variables}
Let $m \in \nat^+$, let $K$ be a field such that $p(K(X))\leq 2^m$ and let $F/K(\!(t_1,t_2)\!)$ be a finite field extension.
Set $V=\{v\in \Omega_F \mid  p(F^v)>2^m \}$.
Then $V$ is finite and, for every $v\in V$, we have that $Fv/K$ is a nonruled function field in one variable.
\end{prop}
\begin{proof}
For any field $L$ we set $p'(L)=\min\{p(L),s(L)+1\}\in\nat\cup\{\infty\}$.
Consider a finite field extension $L/K$.
By \cite[Corollary~4.6]{BGVG14}, the hypothesis on $K$ implies that $p(L(X))\leq 2^m$.
Moreover, if $L$ is nonreal, then it follows by \cite[Theorem 3.5]{BVG09} that $s(L(X))=s(L)<2^m$. In any case, we obtain that $p'(L)\leq p'(L(X))\leq 2^m$.

Consider $v\in\Omega_F$. Note that $p(F^v)=p'(Fv)$.
If the extension $Fv/K$ is either algebraic or a ruled function field in one variable, then  we obtain that $p(F^v)=p'(Fv)\leq 2^m$.
If $Fv$ carries a complete $\zz$-valuation $w$ whose residue field $(Fv)w$ is a finite extension of $K$, then $p(F^v)=p'(Fv)=p'((Fv)w)\leq 2^m$,

In view of \Cref{residues}, this shows that, for every $v\in V$, the extension $Fv/K$ is a nonruled function field in one variable. 
We conclude by \Cref{fin val} that the set $V$ is finite.
\end{proof}

\begin{thm}\label{Pythagoras FPS in 2 variable}
Let $n \in \nat$, let $K$ be a $\mc{P}_n$-field and let $F/K(\!(t_1,t_2)\!)$ be a finite field extension.
Then $$p(F)\leq 2^{n+1}+1\,.$$
\end{thm}

\begin{proof}
We set $V=\{v\in \Omega_F \mid  p(F^v)>2^{n+1} \}$.
Then $V$ is finite, by  \Cref{FPS in 2 variables}. 
Let $S$ be the integral closure of $\rfps{K}{t_1,t_2}$ in $F$. 
Recall that $S$ is a $2$-dimensional noetherian complete local ring by \cite[Theorem 4.3.4]{HS06}, and that $F$ is the fraction field of $S$.
We may assume that $\car F=0$, since otherwise we trivially have that $p(F)\leq 3\leq 2^{n+1}+1$.
Hence $\car K=0$ and thus the residue field of $S$ has characteristic $0$ as well.
It follows by \cite[Corollary 4.7]{HHK15} that quadratic forms in at least $3$ variables satisfy the local-global principle with respect to $\Omega_F$.
Since $2^{n+1}+1\geq 3$, we conclude by \Cref{C:characterise-criterion} that $V$ characterises sums of $2^{n+1}$ squares~$F$.
In particular, if $V=\emptyset$, then $p(F)\leq 2^{n+1}$.

Assume now that $V\neq \emptyset$.
Consider $v \in V$.
By \Cref{FPS in 2 variables}, $Fv/K$ is a function field in one variable.
If $p(Fv)>2^{n+1}$, then $Fv$ is $(n+1)$-effective by the hypothesis that $K$ is a $\mc{P}_n$-field.
Assume now that $p(Fv)\leq 2^{n+1}$.
Then, since $2^{n+1}<p(F^v)\leq p(Fv)+1$, we obtain that $p(F^v)=2^{n+1}+1$ and that $s(Fv)=p(Fv)=2^{n+1}$. 
By  \Cref{EX:s=p-n-effective}, it follows that $Fv$ is $(n+1)$-effective.
Hence $Fv$ is $(n+1)$-effective for every $v\in V$.
We conclude by \Cref{T:example0} that
$F$ is $(n+1)$-effective.
In particular $p(F)\leq 2^{n+1}+1$, by \Cref{Bound for k-effective}.
\end{proof}

\begin{cor}\label{C:final bound}
Let $n\in\nat$ be such that $p(E)\leq 2^{n+1}$ holds for every function field in one variable $E/K$. 
Let $r\in\nat$ and let $F$ be a finite field extension of $K(\!(t_1)\!)\dots(\!(t_r)\!)(\!(t_{r+1},t_{r+2})\!)$.
Then $p(F)\leq 2^{n+1}+1$.
\end{cor}

\begin{proof}
The hypothesis implies that $K$ is a $\mc P_{n}$-field; see \Cref{FPS}. 
By an iterated application of \Cref{T:example2}, we obtain that $K(\!(t_1)\!)\ldots(\!(t_r)\!)$ is a $P_{n}$-field as well.
Hence $p(F)\leq 2^{n+1}+1$, by \Cref{Pythagoras FPS in 2 variable}.
\end{proof}

\begin{rem}
As mentioned in the introduction, for $K=\rr$, $r=0$ and $n=0$, the bound in \Cref{C:final bound} gives an alternative proof of \cite[Theorem 5.1]{Hu15}.
\Cref{C:final bound} also applies when $K$ is an extension of transcendence degree $n$ of $\rr$ (or of transcendence degree $n-1\geq 1$ of $\qq$) and  gives that $p(F)\leq 2^{n+1}+1$ for any field $F$ as in the statement.
When $n\geq 2$, this is an improvement compared to the bound $p(F)<2^{n+2}$, which one could derive from \cite[Corollary 4.7]{Hu17} by using \cite[Theorem 3.5]{BVG09}.
\end{rem}

\begin{qu}
Is $\rr(\!(t_1,t_2)\!)$ a $\mc{P}_1$-field?
\end{qu}

\section*{Acknowledgments}
The authors wish to express their gratitude to David Grimm and Yong Hu for inspiring discussions and helpful comments, in particular in the context of \Cref{dim 2}.
They further want to thank the referee for their careful reading and detailed corrections.

\bibliographystyle{plain}

\end{document}